\documentclass[12pt]{amsart}
\usepackage{amsmath}
\usepackage{amscd}
\usepackage{amssymb}
\usepackage{amsfonts}
\usepackage{graphicx}
\usepackage{wrapfig}
\newtheorem{theorem}{Theorem}[section]
\newtheorem{lemma}[theorem]{Lemma}
\newtheorem{corollary}[theorem]{Corollary}
\newtheorem{proposition}[theorem]{Proposition}

\theoremstyle{definition}

\theoremstyle{remark}
\newtheorem{remark}[theorem]{Remark}
\numberwithin{equation}{section}

\usepackage[utf8]{inputenc}
\usepackage{tikz}
\usetikzlibrary{patterns}
\pgfdeclarepatternformonly{my dots}{\pgfqpoint{-1pt}{-1pt}}{\pgfqpoint{2.5pt}{2.5pt}}{\pgfqpoint{6pt}{6pt}}%
{
	\pgfpathcircle{\pgfqpoint{0pt}{0pt}}{.5pt}
	\pgfpathcircle{\pgfqpoint{7pt}{7pt}}{.5pt}
	\pgfusepath{fill}
}
\begin{document}
\title[Gradient estimates for a nonlinear parabolic equation with boundary]
{Gradient estimates for a nonlinear parabolic equation with Dirichlet boundary condition}

\author{Xuenan Fu}
\address{Department of Mathematics, Shanghai University, Shanghai 200444, China}
\email{xuenanfu97@163.com}

\author{Jia-Yong Wu}
\address{Department of Mathematics, Shanghai University, Shanghai 200444, China}
\email{wujiayong@shu.edu.cn}

\let\thefootnote\relax\footnotetext{2020 \textit{Mathematics Subject Classification}.
Primary 58J35; Secondary 35B53}

\keywords{Smooth metric measure space, Bakry-Emery Ricci curvature, manifold with boundary,
gradient estimate, Liouville theorem}

\begin{abstract}
In this paper, we prove Souplet-Zhang type gradient estimates for a nonlinear parabolic equation
on smooth metric measure spaces with the compact boundary under the Dirichlet boundary condition when
the Bakry-Emery Ricci tensor and the weighted mean curvature are both bounded below. As an
application, we obtain a new Liouville type result for some space-time functions on such smooth metric
measure spaces. These results generalize previous linear equations to a nonlinear case.
\end{abstract}
\maketitle

\section{Introduction}\label{Int1}

In \cite{Yau}, Yau proved gradient estimates for harmonic functions on complete manifolds
with the Ricci curvature bounded below. In \cite{[SZ]}, Souplet and Zhang generalized Yau's
gradient estimate to the heat equation by adding a necessary logarithmic correction term.
Recently, Kunikawa and Sakurai \cite{KS20} extended Yau and Souplet-Zhang type gradient
estimates to the case of manifolds with the boundary under some Dirichlet boundary condition.
Shortly later, H. Dung, N. Dung and Wu \cite{DDW} generalized Kunikawa-Sakurai results
to the $f$-Laplacian equation and the $f$-heat equation on smooth metric measure spaces with
the compact boundary under some Dirichlet boundary condition; see also N. Dung and Wu \cite{{DW}}
for further generalizations in this direction.

From above, we see that previous gradient estimates on manifolds with the boundary only focused
on linear equations. In this paper we will investigate Souplet-Zhang type gradient estimates for
a \emph{nonlinear} parabolic equation on smooth metric measure spaces with the compact boundary
under some Dirichlet boundary condition.

Let $(M,g)$ be an $n$-dimensional
complete Riemannian manifold, and let $f$ be a smooth potential function on $M$. The triple
$(M,g,e^{-f}dv_g)$ is called a smooth metric measure space, where $dv_g$ is the Riemannian volume
element of metric $g$ and $e^{-f}dv_g$ is the weighted volume element. On $(M,g,e^{-f}dv_g)$,
Bakry and Emery \cite{[BE]} introduced the Bakry-Emery Ricci tensor
\[
\mathrm{Ric}_f:=\mathrm{Ric}+\mathrm{Hess}\,f,
\]
where $\mathrm{Ric}$ is the Ricci tensor of $(M,g)$ and $\mathrm{Hess}$ is the
Hessian with respect to metric $g$. The Bakry-Emery Ricci tensor is a natural
generalization of Ricci curvature on Riemannian manifolds, which plays an
important role in smooth metric measure spaces, see for example \cite{LV},
\cite{[WW]}, \cite{Wj13}, \cite{Wj19}, \cite{[WW1]}, \cite{[WW2]} and
references therein. In particular, a smooth metric measure space satisfying
\[
\mathrm{Ric}_f=\lambda g
\]
for some $\lambda\in \mathbb{R}$ is called a gradient Ricci soliton,
which can be considered as a natural generalization of an Einstein manifold.
The gradient Ricci soliton is called shrinking, steady, or expanding, if $\lambda>0$,
$\lambda=0$, or $\lambda<0$, respectively. Gradient Ricci solitons play
important roles in the Ricci flow and Perelman's resolution of the Poincar\'e
conjecture, see for example \cite{Cao}, \cite{[Ham]}, \cite{[P1]}, \cite{[P2]}, \cite{[P3]}
and references therein. Our interest of this paper in the Bakry-Emery Ricci
tensor is largely due to a nonlinear parabolic equation related to gradient Ricci solitons.

On smooth metric measure space $(M,g,e^{-f}dv_g)$,
the $f$-Laplacian (also called the weighted Laplacian or the Witten Laplacian) is defined by
\[
\Delta_f:=\Delta-\nabla f\cdot\nabla,
\]
which is usually linked with the Bakry-Emery Ricci tensor by the
following generalized Bochner-Weitzenb\"ock formula
\[
\Delta_f|\nabla u|^2=2|\mathrm{Hess}\,u|^2+2\langle\nabla\Delta_f u, \nabla u\rangle+2\mathrm{Ric}_f(\nabla u, \nabla u)
\]
for any $u\in C^\infty(M)$. When $\mathrm{Ric}_f$ is bounded below and $f$ is assumed to
some restriction, there have been a lot of work in this direction; see for example
\cite{[MuWa]}, \cite{[MuWa2]}, \cite{[WW]}, \cite{Wj13}, \cite{[WW1]} and \cite{[WW2]},
etc.. On $(M,g,e^{-f}dv_g)$ with the compact boundary $\partial M$, the associated
$f$-mean curvature (also called weighted mean curvature) is defined by
\[
\mathrm{H}_f:=\mathrm{H}-\nabla f\cdot\nu,
\]
where $\nu$ is the unit outer normal vector to $\partial M$ and $\mathrm{H}$ is the mean curvature of
$\partial M$ with respect to $\nu$. When $f$ is constant, the above notations all return to
the manifold case.

We now study Souplet-Zhang gradient estimates for positive solutions to the nonlinear
parabolic equation
\begin{equation}\label{nonlieq}
u_t=\Delta_fu+au\ln u,
\end{equation}
where $a\in \mathbb{R}$, on $(M,g,e^{-f}dv)$ with compact boundary $\partial M$ under some Dirichlet
boundary condition. It is known that all solutions to its Cauchy problem exist for all
time. Equation \eqref{nonlieq} is closely related to gradient Ricci solitons and weighted
log-Sobolev constants; see for example \cite{Ma} and \cite{Wj} for detailed explanations.
By improving the argument of \cite{DDW} especially for the boundary case, we prove
Souplet-Zhang type gradient estimates for positive bounded solutions to the equation
\eqref{nonlieq} without any assumption on the potential function $f$.
\begin{theorem}\label{NL}
Let $(M,g,e^{-f}dv)$ be an $n$-dimensional smooth metric measure space with the compact boundary.
Assume that
\[
{\rm Ric}_f\ge-(n-1)K\quad \mathrm{and} \quad \mathrm{H}_f\ge-L
\]
for some non-negative constants $K$ and $L$. Let $0<u\le B$ for some positive constant $B$ be a solution
to \eqref{nonlieq} in
\[
Q_{R, T}(\partial M):=B_R(\partial M)\times[-T, 0]\subset M\times (-\infty,\infty),
\]
where $B_R(\partial M):=\{x|d(x,\partial M)<R\}$ and $T>0$. If $u$ satisfies the Dirichlet
boundary condition (that is, $u(x,t)|_{\partial M}$ is constant for each time slice $t\in[-T, 0]$),
\[
u_\nu\ge 0 \quad \mathrm{and} \quad u_t\le a u\ln u
\]
over $\partial M\times[-T, 0]$, then there exists a constant $c(n)$ depending on $n$ such that
\[
\sup\limits_{Q_{R/2, T/2}(\partial M)}\frac{|\nabla u|}{u}
\le c(n)\left(\frac{\sqrt{D}+1}{R}+\frac{1}{\sqrt{T}}+\sqrt{K}+L+\sqrt{E}\right)\sqrt{1+\ln\frac{B}{u}},
\]
where $D:=1+\ln B-\ln(\inf_{Q_{R,T}(\partial M)}u)$ and
$E:=\max\{(n-1)K+\tfrac a2,0\}+\max\{\tfrac a3(1+\ln B),0\}$.
\end{theorem}
\begin{remark}
When $a=0$ and $L=0$, the theorem returns to linear cases \cite{KS20} and \cite{DDW}. Notice
that our weighted mean curvature assumption here could be bounded below by a non-negative
constant (not just non-negative) and hence it indicates that our result is suitable to a
little more general setting.
\end{remark}
In Theorem \ref{NL}, ${\rm Ric}_f\ge-(n-1)K$ means that the infimum of ${\rm Ric}_f$
on the unit tangent bundle in the interior of $M$ is more than $-(n-1)K$, while $\mathrm{H}_f\ge-L$
means that boundary $\partial M$ has some weak convex property. In previous works, Kunikawa and Sakurai
\cite{KS20} proved Souplet-Zhang type gradient estimates for the heat equations under some Dirichlet
boundary condition; H. Dung, N. Dung and Wu \cite{DDW} generalized their result to the
$f$-heat equation; N. Dung and Wu \cite{DW} further studied the case of the $f$-heat equation.
Now we further generalize these linear cases to a nonlinear parabolic equation.

The second author \cite{Wj} ever proved that Souplet-Zhang gradient estimates for positive
bounded solutions to equation \eqref{nonlieq} only hold for radius $R\ge2$ of a geodesic
ball on non-compact manifolds without boundary, but the gradient estimates in Theorem
\ref{NL} could hold for any radius $R>0$. Because in the proof of our case we may apply
a new weighted Laplacian comparison on neighborhoods of the boundary (see Theorem \ref{Lapcom}
in Section \ref{sec2}) instead of the Wei-Wylie comparison \cite{[WW]}.

The proof of Theorem \ref{NL} employs the arguments of \cite{KS20}, \cite{DDW} and \cite{Wj}.
The outline of the proof is as follows. In the interior of the manifold, as in \cite{Wj},
we essentially use the standard Souplet-Zhang argument for the nonlinear parabolic equation.
On the boundary of the manifold, adapting arguments of \cite{KS20} and \cite{DDW},
we apply a derivative equality (see Proposition \ref{Reilly} in Section \ref{sec2}) to
prove the desired estimate. Compared with the previous proof, here we need to carefully deal
with an extra nonlinear term.

Finally we would like to mention that there exist some geometric results on manifolds with the
Neumann boundary condition, e.g. \cite{Ch}, \cite{WaJ}, \cite{[LY]} and \cite{Ol}.

As a consequence of Theorem \ref{NL}, we get the following Liouville type result.
\begin{corollary}\label{liou2}
Let $(M, g, e^{-f}dv)$ be a complete smooth metric measure space with compact boundary
satisfying
\[
{\rm Ric}_f\ge 0 \quad \mathrm{and} \quad \mathrm{H}_f\ge 0.
\]
Let $1\le u\le B$ for some constant $B$ be an ancient solution
to \eqref{nonlieq} satisfying $a\le 0$ with the Dirichlet boundary condition. If
\[
u_\nu\ge 0 \quad \mathrm{and} \quad u_t\le au\ln u
\]
over $\partial M\times(-\infty, 0]$, then $u$ is constant if $a=0$; $u(t)=\exp\{ce^{at}\}$
for some constant $c\le0$ if $a<0$.
\end{corollary}

The rest of this paper is organized as follows. In Section \ref{sec2}, we will list
some basic results about smooth metric measure spaces with the compact boundary, which
will be used in our proof of Theorem \ref{NL}. In particular, we introduce the weighted
Laplacian comparison, the derivative equality, the Bochner type formula and the space-time
cut-off function. In Section \ref{sec3}, we will adopt arguments of \cite{Wj} and
\cite{KS20} to prove Theorem \ref{NL}. In the end we will apply Theorem \ref{NL}
to prove Corollary \ref{liou2}.

\vspace{.1in}

\textbf{Acknowledgement}.
The authors thank the anonymous referee for making useful comments and
pointing out some errors in an earlier version of the paper.
This work was partially supported by the Natural Science Foundation of Shanghai (17ZR1412800).


\section{Background}\label{sec2}
In this section, we mainly recall some basic results about smooth metric measure spaces with the
compact boundary, which will be used in the proof of our result. We refer the interested reader
to \cite{Sa17, Sakurai17, Saku} for more results. On an $n$-dimensional smooth metric measure
space $(M^n, g, e^{-f}dv)$ with the boundary $\partial M$, for any point $x\in M$, we let
\[
\rho(x)=\rho_{\partial M}(x)=d(x, \partial M),
\]
denote the distance function from the boundary $\partial M$. From the argument of \cite{Sa17},
we may assume that distance function $\rho$ is smooth outside of the cut locus for the
boundary ${\rm Cut}(\partial M)$. In \cite{WZZ}, Wang, Zhang and Zhou proved weighted
Laplacian comparisons for the distance function on smooth metric measure spaces with
the boundary under some assumptions (see also \cite{Saku}). Later, Sakurai \cite{Sakurai17}
proved a general comparison result, which plays an important role in the proof of
Theorem \ref{NL}.

\begin{theorem}\label{Lapcom}
Let $(M^n, g, e^{-f}dv)$ be an $n$-dimensional complete smooth metric measure space with
the compact boundary $\partial M$. Assume that
\[
{\rm Ric}_f\ge-(n-1)K\quad \mathrm{and} \quad \mathrm{H}_f\ge-L
\]
for some constants $K\ge0$ and $L\in\mathbb{R}$. Then
\[
\Delta_f\rho(x)\le (n-1)KR+L
\]
for all $x\in B_R(\partial M)$ outside of ${\rm Cut}(\partial M)$, where
$B_R(\partial M):=\{x|d(x,\partial M)<R\}$.
\end{theorem}

\begin{remark}
This comparison result was used in the proof of gradient estimatess for linear
equations in \cite{DDW} and \cite{DW}. We would like to point out that this
comparison theorem holds for all $R>0$, which is quite different from the
Wei-Wylie comparison on smooth metric measure spaces without boundary. Indeed,
the Wei-Wylie comparison theorem requires the radius of geodesic ball $R\ge R_0$
for some constant $R_0>0$; see Theorem 3.1 of \cite{[WW]}.
\end{remark}

Next, we recall a useful derivative equality, which can be derived from the proof of
the weighted Reilly formula in \cite{[MD]}; see also (24) in Appendix of \cite{[CMZ]}.
The present derivative equality can be regarded as a weighted version of the classical
setting \cite{Rei} and will be used in the proof of gradient estimate for the boundary case.
\begin{proposition}\label{Reilly}
Let $(M^n, g, e^{-f}dv)$ be an $n$-dimensional complete smooth metric measure space with
the compact boundary $\partial M$. For any $u\in C^\infty(M)$,
\begin{equation*}
\begin{aligned}
\frac{1}{2}\left(|\nabla u|^2\right)_\nu&
=u_\nu\left[\Delta_f u-\Delta_{\partial M, f}\left(u|_{\partial M}\right)-\mathrm{H}_f u_\nu\right]
+g_{\partial M}(\nabla_{\partial M}(u|_{\partial M}), \nabla_{\partial M}u_\nu)\\
&\quad-{\rm II}(\nabla_{\partial M}(u|_{\partial M}), \nabla_{\partial M}(u|_{\partial M})),
\end{aligned}
\end{equation*}
where $\nu$ is the outer unit normal vector to $\partial M$, and ${\rm II}$ is the second
fundamental form of $\partial M$ with respect to $\nu$.
\end{proposition}

Meanwhile, we need an important Bochner type formula for equation \eqref{nonlieq} in the
proof of our result, which was similarly discussed in \cite{DD}.
Let $0<u\le B$ for some positive constant $B$ be a solution
to \eqref{nonlieq} in $Q_{R, T}(\partial M)$. Consider the function
\[
h(x,t):=\sqrt{1+\ln\frac{B}{u(x,t)}}=\sqrt{\ln\frac{A}{u(x,t)}}
\]
in $Q_{R, T}(\partial M)$. Obviously,
\[
A=Be\quad \mathrm{and} \quad h(x,t)\ge 1.
\]
As in \cite{DD}, we have the following Bochner type formula, which holds
on smooth metric measure spaces without any assumption on $f$.

\begin{lemma}\label{lemheat1}
Under the same assumptions of Theorem \ref{NL}, for any $(x,t)\in Q_{R,T}(\partial M)$, the function
$w=|\nabla h|^2$ satisfies
\begin{equation*}
\begin{aligned}
\Delta_f w-{w_t}&\ge 2(2h-h^{-1})\langle \nabla w,\nabla h\rangle+2(2+h^{-2}){w^2}\\
&\quad-\left[\max\big\{2(n-1)K+a,\,0\big\}+h^{-2}\max\big\{ a\ln A,\,0 \big\}\right]w,
\end{aligned}
\end{equation*}
where $A:=Be$.
\end{lemma}

\begin{proof}[Proof of Lemma \ref{lemheat1}]
The proof is essentially the same as Lemma 2.1 in \cite{DD} and we include it for the sake of completeness.
Since $u=Ae^{-h^2}$, we compute that
\[
u_t=-2Ahe^{-h^2}h_t, \quad \nabla u=-2Ahe^{-h^2}\nabla h
\]
and
\[
\Delta_f u=-2Ahe^{-h^2}\Delta_f h-2A(1-2h^2)|\nabla h|^2e^{-h^2}.
\]
By \eqref{nonlieq}, we get that
\begin{equation*}
\begin{aligned}
-2Ahe^{-h^2} h_t&=u_t\\
&=\Delta_fu+au\ln u\\
&=-2Ahe^{-h^2}\Delta_f h-2A(1-2h^2)|\nabla h|^2e^{-h^2}-aA(h^2-\ln A)e^{-h^2},
\end{aligned}
\end{equation*}
which is equivalent to
\begin{equation}\label{hequva}
h_t=\Delta_f h+(h^{-1}-2h)|\nabla h|^2+\frac a2\left(h-\ln A\cdot h^{-1}\right).
\end{equation}
On the other hand, by the generalized Bochner-Weitzenb\"ock formula and the curvature
assumption ${\rm Ric}_f\ge-(n-1)K$, we have
\begin{equation*}
\begin{split}
\Delta_f|\nabla h|^2=&2|\mathrm{Hess}\,
h|^2+2\langle\nabla\Delta_f h, \nabla h\rangle+2\mathrm{Ric}_f(\nabla h, \nabla h)\\
\ge&2\langle\nabla\Delta_f h, \nabla h\rangle-2(n-1)K|\nabla h|^2
\end{split}
\end{equation*}
and hence
\[
\Delta_fw-w_t\ge2\langle\nabla\Delta_f h, \nabla h\rangle-2(n-1)Kw-w_t.
\]
Combining this with \eqref{hequva} and using the definition of $w$, we get
\begin{equation}
\begin{split}\label{compineq}
\Delta_fw-w_t&\ge2\langle\nabla h_t, \nabla h\rangle+2\left\langle\nabla[(2h-h^{-1})w], \nabla h\right\rangle\\
&\quad+a\left\langle\nabla(\ln A\cdot h^{-1}-h), \nabla h\right\rangle-2(n-1)Kw-w_t\\
&=2\langle\nabla h_t, \nabla h\rangle+2\left\langle\nabla(2h-h^{-1}), \nabla h\right\rangle w
+2(2h-h^{-1})\langle\nabla w, \nabla h\rangle\\
&\quad-a\ln A\cdot h^{-2}w-aw-2(n-1)Kw-w_t.
\end{split}
\end{equation}
Notice that
\[
2\langle\nabla h_t, \nabla h\rangle=w_t \quad\mathrm{and} \quad\nabla(2h-h^{-1})=(2+h^{-2})\nabla h.
\]
Therefore \eqref{compineq} becomes
\[
\Delta_fw-w_t\ge 2(2+h^{-2})w^2+2(2h-h^{-1})\langle\nabla w, \nabla h\rangle-\left[2(n-1)K+a\right]w-a\ln A\cdot h^{-2}w
\]
and the result follows.
\end{proof}

To prove Theorem \ref{NL}, we also need a space-time cut-off function, which was ever used in
\cite{[LY]}, \cite{[SZ]}, \cite{KS20} and \cite{DDW}.
	\begin{lemma}\label{lemmcut2}
Let $(M^n, g, e^{-f}dv)$ be an $n$-dimensional complete smooth metric measure space
with the compact boundary $\partial M$. There exists a smooth cut-off function
$\psi=\psi(\rho,t)\equiv\psi\left(\rho_{\partial M}(x),t \right)$ supported in
$Q_{R,T}(\partial M)$ and a constant $C_\varepsilon>0$ depending only on $0<\varepsilon<1$
such that
\begin{itemize}
			\item[(i)] $0\le \psi(\rho,t) \le 1$ in $Q_{R,T}(\partial M)$ and $\psi(\rho,t)=1$
in $Q_{R/2, T/2}(\partial M)$.
			\item[(ii)] $\psi$ is decreasing as a radial function of parameter $r$.
			\item[(iii)]
\[
\frac{|\psi_t|}{\psi^{1/2}}\le\frac{C}{T},\quad
|\psi_\rho|\le \frac{C_{\varepsilon}{{\psi}^{\varepsilon }}}{R}
\quad\mathrm{and}\quad
|\psi_{\rho\rho}|\le \frac{C_{\varepsilon}{{\psi}^{\varepsilon }}}{{R^2}},
\]
where $C>0$ is a universal constant.
\end{itemize}
\end{lemma}


\section{Souplet-Zhang gradient estimate}\label{sec3}
In this section, we will apply some results of Section \ref{sec2} to prove Theorem
\ref{NL} by adapting arguments of \cite{KS20}, \cite{DDW} and \cite{Wj}.
It is worth to point out that we need carefully to discuss the boundary case.
\begin{proof}[Proof of Theorem \ref{NL}]
Let $w$ be the function in Lemma \ref{lemheat1} and $\psi$ denote the cut-off function
in Lemma \ref{lemmcut2}. Our aim is to estimate $(\Delta_f-\partial_t)(\psi w)$
and carefully analyze the result at a space-time point where the function $\psi w$
attains its maximum. Assume that the space-time maximum of $\psi w$ is reached at
some point $(x_1,t_1)$ in ${Q_{R,T}(\partial M)}$. We will prove the Souplet-Zhang
gradient estimate according to two cases: $x_1\not\in\partial M$ and $x_1\in\partial M$.

\textbf{Case 1: }If $x_1\not\in\partial M$, we may assume without loss of generality
that $x_1\notin {\rm Cut}(\partial M)$ by the Calabi's argument.
Since at $(x_1,t_1)$, we have
\[
\Delta_f(\psi w)\leq0,\quad(\psi w)_t\geq0
\quad \mathrm{and}\quad\nabla(\psi w)=0.
\]
Using the above properties and Lemma \ref{lemheat1}, at $(x_1,t_1)$, we have
\begin{equation*}
\begin{aligned}
0&\ge \Delta_f(\psi w)-(\psi w)_t\\
&=\Delta_f\psi\cdot w+2\langle\nabla w, \nabla\psi\rangle+\psi(\Delta_f w-w_t)-\psi_tw\\
&\ge \Delta_f\psi\cdot w-2\frac{|\nabla \psi|^2}{\psi}w-\psi_t\cdot w
+2(2h-h^{-1})\psi\langle \nabla w,\nabla h \rangle+2(2+h^{-2})\psi{w^2}\\
&\quad-\left[\max\big\{2(n-1)K+a,\,0\big\}+h^{-2}\max \big\{ a\ln A,\,0 \big\}\right]\psi w\\
&=\Delta_f\psi\cdot w-2\frac{|\nabla \psi|^2}{\psi}w-\psi_t\cdot w
-2(2h-h^{-1})\langle\nabla\psi,\nabla h \rangle w+2(2+h^{-2})\psi{w^2}\\
&\quad-\left[\max\big\{2(n-1)K+a,\,0\big\}+h^{-2}\max\big\{a\ln A,\,0 \big\}\right]\psi w.
\end{aligned}
\end{equation*}
This inequality can be written as
\begin{equation*}
\begin{aligned}
2\psi\omega^2\le&\frac{2h^2}{1+2h^2}\frac{|\nabla \psi|^2}{\psi}w
-\frac{2h(1-2h^2)}{1+2h^2}\langle\nabla\psi,\nabla h \rangle w+\frac{h^2}{1+2h^2}\psi_t\cdot w
-\frac{h^2}{1+2h^2}\Delta_f\psi\cdot w\\
&\quad+\left[\frac{h^2}{1+2h^2}\max\big\{2(n-1)K+a,\,0\big\}+\frac{1}{1+2h^2}\max\big\{a\ln A,\,0\big\}\right]\psi w
\end{aligned}
\end{equation*}
at $(x_1,t_1)$. Since $h\ge 1$, and then
\[
0<\frac{h^2}{1+2h^2}\le \frac 12\quad \mathrm{and} \quad 0<\frac{1}{1+2h^2}\le \frac 13,
\]
so the above inequality can be further simplified as
\begin{equation}\label{lefor}
2\psi\omega^2\le\frac{|\nabla \psi|^2}{\psi}w
-\frac{2h(1-2h^2)}{1+2h^2}\langle\nabla\psi,\nabla h \rangle w+\tfrac 12|\psi_t|\cdot w
-\frac{h^2}{1+2h^2}\Delta_f\psi\cdot w+E\psi w
\end{equation}
at $(x_1,t_1)$, where $E:=\max\{(n-1)K+\tfrac a2,\,0\}+\max\{\tfrac{a}{3}\ln A,\,0\}$.

Below we will apply Lemma \ref{lemmcut2} to estimate upper bounds for each term of
the right-hand side of \eqref{lefor}. We remark that for any real numbers $M$ and $N$,
the Young's inequality
\[
C_1C_2\leq\frac{|C_1|^p}{p}+\frac{|C_2|^q}{q},\quad \forall\,\,\, p,q>0
\,\,\,\mathrm{with}\,\,\, \frac1p+\frac1q=1,
\]
where $C_1,C_2\in \mathbb{R}$, will be repeatedly used in the following estimates.
We let $c$ denote a constant may depending on $n$ whose value may change from line to line.
First,
\begin{equation}
\begin{aligned}\label{term1}
\frac{|\nabla\psi|^2}{\psi}w
&=\psi^{1/2}w\cdot\frac{|\nabla\psi|^2}{\psi^{3/2}}\\
&\le\frac{1}{5}\psi\omega^2+c
\left(\frac{|\nabla\psi|^2}{\psi^{3/2}}\right)^2\\
&\le\frac{1}{5}\psi w^2+\frac{c}{R^4}.
\end{aligned}
\end{equation}
Second,
\begin{equation}
\begin{aligned}\label{term2}
-\frac{2h(1-2h^2)}{1+2h^2}\langle\nabla\psi,\nabla h \rangle w
&\le2h|\nabla\psi||\nabla h|w\\
&=2h\frac{|\nabla\psi|}{\psi^{3/4}}(\psi w^2)^{3/4}\\
&\le\frac{1}{5}\psi w^2+h^4\frac{|\nabla\psi|^3}{\psi^3}\\
&\le\frac{1}{5}\psi w^2+\frac{cD^2}{R^4},
\end{aligned}
\end{equation}
where $D:=\ln A-\ln(\inf_{Q_{R,T}(\partial M)}u)$.
Third,
\begin{equation}
\begin{aligned}\label{term3}
\frac 12|\psi_t|\cdot w&=\frac 12\frac{|\psi_t|}{\psi^{1/2}}\cdot\psi^{1/2}w\\
&\le\frac{1}{5}\psi w^2+c\frac{|\psi_t|^2}{\psi}\\
&\le\frac{1}{5}\psi w^2+\frac{c}{T^2}.
\end{aligned}
\end{equation}
Fourth, we will apply Theorem \ref{Lapcom} to the following term
\begin{equation}\label{term4}
\begin{aligned}
-\frac{h^2}{1+2h^2}\Delta_f\psi\cdot w&=-\frac{h^2}{1+2h^2}(\psi_\rho\Delta_f\rho+\psi_{\rho\rho}|\nabla\rho|^2)w\\
&\le \frac{h^2}{1+2h^2}(|\psi_\rho|(n-1)KR+|\psi_\rho|L+|\psi_{\rho\rho}|)w\\
&\le\frac{|\psi_{\rho\rho}|}{2\psi^{1/2}}\psi^{1/2}w+\frac{(n-1)KR+L}{2}\psi^{1/2}w\frac{|\psi_\rho|}{\psi^{1/2}}\\
&\le\frac{1}{5}\psi w^2+c\left[\left(\frac{|\psi_{\rho\rho}|}{\psi^{1/2}}\right)^2
+(K^2R^2+L^2)\left(\frac{|\psi_\rho|}{\psi^{1/2}}\right)^2\right]\\
&\le \frac{1}{5}\psi w^2+\frac{c}{R^4}+cK^2+c\frac{L^2}{R^2}\\
&\le \frac{1}{5}\psi w^2+\frac{c}{R^4}+cK^2+cL^4.
\end{aligned}
\end{equation}
Fifth,
\begin{equation}\label{term5}
E\psi w\leq\frac{1}{5}\psi w^2+cE^2.
\end{equation}
Substituting \eqref{term1}-\eqref{term5} into the right hand side of \eqref{lefor},
we have
\[
\psi w^2\le c\left(\frac{D^2+1}{R^4}+L^4+\frac{1}{T^2}+K^2+E^2\right)
\]
at $(x_1,t_1)$. Then for all $(x,t)\in Q_{R/2,T/2}(\partial M)$, we have $\psi(x,t)\equiv 1$
and therefore
\begin{equation*}
\begin{aligned}
w^2(x,t)&=\psi(x,t)w^2(x,t)\\
&\le\psi(x_1,t_1)w^2(x_1,t_1)\\
&\le c\left(\frac{D^2+1}{R^4}+L^4+\frac{1}{T^2}+K^2+E^2\right).
\end{aligned}
\end{equation*}
Since $w=|\nabla h|^2$, then
\[
|\nabla h|(x,t)\le c\left(\frac{\sqrt{D}+1}{R}+\frac{1}{\sqrt{T}}+\sqrt{K}+L+\sqrt{E}\right)
\]
for all $(x,t)\in Q_{R/2,T/2}(\partial M)$. The desired gradient estimate follows by
using the equality
\[
|\nabla h|=\frac{|\nabla u|}{2u\sqrt{\ln(A/u)}}.
\]

\textbf{Case 2:} If $x_1\in \partial M$, the gradient estimate of Theorem \ref{NL} still holds.
Indeed at the maximum point $(x_1,t_1)$ of $\psi w$, we have $(\psi w)_\nu\ge 0$ and hence
\[
\psi_\nu w+\psi w_\nu=\psi w_\nu\ge0,
\]
which implies
\[
w_\nu\ge0.
\]
Since $w=|\nabla h|^2$, and $h=\sqrt{\log(A/u)}$
satisfies the Dirichlet boundary condition, by Proposition \ref{Reilly}, we have
\begin{equation}\label{boundest}
0\le w_\nu=(|\nabla h|^2)_\nu=2h_\nu(\Delta_f h-\mathrm{H}_f h_\nu).
\end{equation}
Notice that since $u$ satisfies the Dirichlet boundary condition, then
$|\nabla u|=u_{\nu}$ and hence
\begin{equation}\label{cond1}
h_\nu=\frac{-u_\nu}{2u\sqrt{\log(A/u)}}=\frac{-|\nabla u|}{2u\sqrt{\log(A/u)}}=-w^{1/2}.
\end{equation}
We also have
\begin{equation}
\begin{aligned}\label{cond2}
	\Delta_fh&={\rm div}\left(\frac{-\nabla u}{2u\sqrt{\log(A/u)}}\right)-\left\langle\nabla f, \nabla h\right\rangle\\
	&=\frac{-\Delta u}{2u\sqrt{\log(A/u)}}-\frac{1}{2}\left\langle\nabla u, \nabla\left(\frac{1}{u \sqrt{\log(A/u)}}\right)\right\rangle+\left\langle\nabla f, \frac{\nabla u}{2u \sqrt{\log(A/u)}}\right\rangle\\
	&=\frac{-\Delta_f u}{2u \sqrt{\log(A/u)}}+\frac{1}{2}\left(\frac{|\nabla u|^2}{u^2\sqrt{\log(A/u)}}-\frac{|\nabla u|^2}{2u^2(\log(A/u))^{3/2}}\right)\\
	&=-\frac{u_t}{2uh}+\frac{a\ln u}{2h}+\left(2h-\frac 1h\right)w.
\end{aligned}
\end{equation}
Substituting \eqref{cond1} and \eqref{cond2} into \eqref{boundest} yields
\[
-\frac{u_t}{2uh}+\frac{a\ln u}{2h}+\left(2h-\frac 1h\right)w+\mathrm{H}_f w^{1/2}\le0
\]
at $(x_1,t_1)$. Since $\partial_tu\le a u\ln u$ over $\partial M\times[-T, 0]$ by our theorem
assumption, then
\[
-\frac{u_t}{2uh}+\frac{a\ln u}{2h}\ge 0
\]
and hence
\[
\left(2h-\frac 1h\right)w+\mathrm{H}_f w^{1/2}\le0
\]
at $(x_1,t_1)$. Since $h\ge 1$, then $(2h-h^{-1})\ge 1$ and therefore we get
\[
w+\mathrm{H}_f w^{1/2}\le0
\]
at $(x_1,t_1)$. This implies
\[
w(x_1,t_1)=0\quad \mathrm{or} \quad w^{1/2}(x_1,t_1)\le L,
\]
where we used the theorem assumption $\mathrm{H}_f\ge-L$. It means that
\[
\psi w\equiv0 \quad\mathrm{or}\quad (\psi w)(x_1,t_1)\le L^2
\]
on $Q_{R,T}(\partial M)$. The former indicates that $u$ is constant and the conclusion
follows; the latter gives that for all $(x,t)\in Q_{R/2,T/2}(\partial M)$,
$\psi(x,t)\equiv 1$ and
\begin{equation*}
\begin{aligned}
|\nabla h|^2(x,t)&=w(x,t)\\
&=\psi(x,t)w(x,t)\\
&\le\psi(x_1,t_1)w(x_1,t_1)\\
&\le L^2,
\end{aligned}
\end{equation*}
which also implies the conclusion by using
\[
|\nabla h|=\frac{|\nabla u|}{2u\sqrt{\ln(A/u)}}.
\]

\end{proof}

In the end, we apply Theorem \ref{NL} to prove Corollary \ref{liou2}.
\begin{proof}[Proof of Corollary \ref{liou2}]
We assume that ${\rm Ric}_f\ge0$ and $\mathrm{H}_f\ge0$
on smooth metric measure space $(M,g,e^{-f}dv)$ with the compact boundary. Let
$1\le u\le B$ be an ancient solution to \eqref{nonlieq} satisfying
$a\le 0$ in $Q_{R, T}(\partial M)$ with the Dirichlet boundary condition. If
\[
u_\nu\ge0 \quad \mathrm{and} \quad u_t\le au\ln u
\]
over $\partial M\times(-\infty, 0]$, then by Theorem \ref{NL}, there exists a
constant $c(n)$ depending on $n$ such that
\[
\sup\limits_{Q_{R/2, T/2}(\partial M)}\frac{|\nabla u|}{u}
\le c(n)\left(\frac{\sqrt{1+\ln B}+1}{R}+\frac{1}{\sqrt{T}}\right)\sqrt{1+\ln B}.
\]
Letting $R\to\infty$ and $T\to\infty$, we obtain $|\nabla u|=0$ and hence
$u(x,t)=u(t)$ only depends on time parameter $t$. Substituting $u(t)$ into
the equation \eqref{nonlieq} we get the following ordinary differential equation
\[
u_t=au\ln u.
\]
If $a=0$, we can get $u_t=0$ and $u$ is constant. If $a<0$, we directly solve
the above equation and obtain
\[
u(t)=\exp\{ce^{at}\}
\]
for some $c\in\mathbb{R}$. Moreover since $u$ is an ancient solution in
$(-\infty, 0]$, we claim that $c\le0$. Indeed if $c>0$ then $u=\exp\{ce^{at}\}$ is
unbounded in $(-\infty, 0]$, which contradicts with our theorem assumption.
\end{proof}


\end{document}